\documentclass[12pt]{amsart}
\usepackage{amsmath,amsthm,amsfonts,latexsym,amssymb,amscd,xcolor}
\usepackage{chemarr}
\usepackage{mathtools}
\usepackage{mathrsfs}
\usepackage{tikz}
\numberwithin{equation}{section}
\newtheorem{thm}[equation]{Theorem} 
\newtheorem{cor}[equation]{Corollary}

\newtheorem*{clm*}{Claim}
\theoremstyle{definition}

\newcommand{\cproof}{\noindent{\it Proof of Claim.}\ } 
\newcommand{\cqed}{\hfill\rule{1.3mm}{3mm}}

\newcommand{\wec}[1]{{\mathbf{#1}}}  
\newcommand{\wrel}[1]{\;#1\;}     

\newcommand{\m}[1]{{\mathbf{\uppercase{#1}}}}

\DeclareMathOperator{\Con}{Con}

\newcommand{\cg}{\mathrm{Cg}}
\newcommand{\C}[1]{{\mathbf{\uppercase{#1}}}}
\newcommand{\Cg}[1]{{\cg}^{\m{#1}}}

\begin{document}

\title[When do abelian tolerances generate abelian congruences?]{What is the weakest idempotent
  Maltsev condition that implies
  that abelian tolerances generate \\ abelian congruences?}

\author{Keith A. Kearnes}
\address[Keith A. Kearnes]{Department of Mathematics\\
University of Colorado\\
Boulder, CO 80309-0395\\
USA}
\email{kearnes@colorado.edu}

\author{Emil W. Kiss}
\address[Emil W. Kiss]{
Lor\'{a}nd E{\"o}tv{\"o}s University\\
Department of Algebra and Number Theory\\
H--1117 Budapest, P\'{a}zm\'{a}ny P\'{e}ter s\'{e}t\'{a}ny 1/c.\\
Hungary}
\email{ewkiss@cs.elte.hu}

\subjclass[2010]{Primary: 08B05;  Secondary: 08A05}
\keywords{abelian congruence, abelian tolerance, Maltsev condition}

\begin{abstract}
We answer the question in the title. In the process,
we correct an error in our AMS Memoir
{\it The Shape of Congruence Lattices}.
\end{abstract}

\maketitle

\section{Introduction}\label{intro}
Our 2013 AMS Memoir \cite{shape}
contains the following claim.
\medskip

\noindent
{\bf Theorem 3.24 of \cite{shape}.}
Assume that $\mathcal V$
satisﬁes a nontrivial idempotent Maltsev condition.
If $\m a\in {\mathcal V}$
has a tolerance $T$ and a congruence $\delta$
such that $\C C(T,T;\delta)$ holds,
then $\C C(\Cg a (T),\Cg a (T);\delta)$ holds.
\medskip

In the case where $\delta=0$, the conclusion of this theorem
asserts that if $T$ is an abelian tolerance, then
it generates an abelian congruence.
\medskip

The recent paper \cite{relative} of the first author
contains the following claim.
\medskip

\noindent
{\bf Theorem 4.5 of \cite{relative}.}
The following are equivalent for a variety $\mathcal V$.
\begin{enumerate}
\item $\mathcal V$ has a weak difference term.
\item Whenever $\m a\in \mathcal V$ and $\alpha\in\Con(\m a)$
  is abelian, the interval $I[0,\alpha]$ consists of permuting
  equivalence relations.
\item Whenever $\m a\in \mathcal V$ and $\alpha\in\Con(\m a)$
  is abelian, the interval $I[0,\alpha]$ is modular.
\item Whenever $\m a\in \mathcal V$ and $\alpha\in\Con(\m a)$,
  there is no pentagon labeled as in
  Figure~\ref{fig8} with $[\alpha,\alpha]=0$.
(No ``spanning pentagon'' in $I[0,\alpha]$.)  
\item Whenever $\m a\in \mathcal V$ and $\alpha\in\Con(\m a)$
  is abelian, there is no pentagon labeled as in
  Figure~\ref{fig8} where $[\alpha,\alpha]=0$ 
  and $\C C(\theta,\alpha;\delta)$.
\begin{figure}[ht]
\begin{center}
\setlength{\unitlength}{1mm}
\begin{picture}(20,33)
\put(0,15){\circle*{1.2}}
\put(10,0){\circle*{1.2}}
\put(10,30){\circle*{1.2}}
\put(20,10){\circle*{1.2}}
\put(20,20){\circle*{1.2}}

\put(10,0){\line(-2,3){10}}
\put(10,0){\line(1,1){10}}
\put(10,30){\line(-2,-3){10}}
\put(10,30){\line(1,-1){10}}
\put(20,20){\line(0,-1){10}}

\put(-4.5,13){$\beta$}
\put(22,9){$\delta$}
\put(22,19){$\theta$}
\put(9,32){$\alpha$}
\put(9,-5){$0$}
\end{picture}
\bigskip

\caption{\sc Forbidden sublattice if both $[\alpha,\alpha]=0$ and
  $\C C(\theta,\alpha;\delta)$ hold.}\label{fig8}
\end{center}
\end{figure}
\end{enumerate}  
\medskip

During the final editing of \cite{relative}, we realized
that Theorem~3.24 of \cite{shape} is not consistent
with Theorem~4.5 of \cite{relative}.
A close examination revealed an error in the proof
of Theorem~3.24 of \cite{shape}.
In this paper, we identify the error and
correct Theorem~3.24 and
correct the results of \cite{shape}
which depend on Theorem~3.24.
We include a short summary of the affected
results at the end of Section~\ref{consequences}.

\section{What is the error?}

As we have written above,
Theorem~3.24 of \cite{shape} is false as stated.
In this section we identify the error in
the proof of Theorem~3.24.
In later sections of this article we will correct Theorem~3.24
and correct the results of \cite{shape}
that depend on Theorem~3.24.

Throughout this article we use the terminology and notation
of the Memoir, \cite{shape}.
The arguments there and here 
rely on the basic properties of the centralizer relation,
which are
enumerated in the following theorem.
\medskip

  \noindent
{\bf Theorem~2.19 of \cite{shape}.}  
  Let $\m a$ be an algebra with tolerances $S$, $S'$, $T$,
  $T'$ and congruences $\alpha$, $\alpha_i$, $\beta$,
  $\delta$, $\delta'$, $\delta_j$.  The following are true.
  \begin{enumerate}
  \item[(1)] {\rm (Monotonicity in the first two variables)}
    If\/ $\C C(S,T;\delta)$ holds and $S'\subseteq S$,
    $T'\subseteq T$, then $\C C(S',T';\delta)$ holds.
  \item [(2)] $\C C(S,T;\delta)$ holds if and only if\/ $\C
    C(\Cg a(S),T;\delta)$ holds.
  \item [(3)] $\C C(S, T; \delta)$ holds if and only if\/
    $\C C(S, \delta\circ T\circ \delta; \delta)$ holds.
  \item [(4)] If\/ $T\cap \delta = T\cap\delta'$, then
    $\C C(S,T;\delta)\wrel{\Longleftrightarrow} \C C(S,T;\delta')$.
  \item [(5)] {\rm (Semidistributivity in the first
      variable)} If\/ $\C C(\alpha_i,T;\delta)$ holds for
    all $i\in I$, then\/ $\C C(\bigvee_{i\in
      I}\alpha_i,T;\delta)$ holds.
  \item[(6)] If\/ $\C C(S,T;\delta_j)$ holds for all $j\in
    J$, then\/ $\C C(S,T;\bigwedge_{j\in J}\delta_j)$ holds.
  \item [(7)] If\/ $T\cap\big(S\circ (T\cap\delta)\circ
    S\big) \subseteq\delta$, then\/ $\C C(S,T;\delta)$
    holds.
  \item [(8)] If\/
    $\beta\wedge\big(\alpha\vee(\beta\wedge\delta)\big)
    \leq\delta$, then\/ $\C C(\alpha,\beta;\delta)$ holds.
  \item [(9)] Let\/ $\m b$ be a subalgebra of\/ $\m a$.
    If\/ $\C C(S, T; \delta)$ holds in $\m a$, then\/ $\C
    C(S|_{\m b}, T|_{\m b}; \delta|_{\m b})$ holds in $\m
    b$.
  \item [(10)] If\/ $\delta'\leq \delta$, then the
    relation\/ $\C C(S, T; \delta)$ holds in $\m a$ if and
    only if\/ $\C C(S/\delta', T/\delta'; \delta/\delta')$
    holds in $\m a/\delta'$.
  \end{enumerate}
\bigskip

The proof of
\cite[Theorem~3.24]{shape} is a little
longer than a page and a half. Just before the half-page mark we read

\medskip

\noindent
{\it Claim~3.25. $\Delta\cap R_i=0$
 is the equality relation for both $i = 1$ and $2$.}
\medskip

\noindent
This claim is proved correctly.
After the statement of Claim~3.26 of \cite{shape}, the proof proceeds
with
\medskip

\noindent
{\it From Claim~3.25 and Theorem 2.19 (8) we get that
$\C C(R_i,\Delta;0)$ holds for $i = 1$
and $2$.}
\medskip

\noindent
The relations $\alpha,\beta,\delta$ that appear
in the statement of Theorem~2.19~(8)
are congruences.
Suppose for a moment that $\Delta$ and $R_i$
from Claim~3.25 are congruences.
Then we could label
$\alpha:=R_i$, $\beta:=\Delta$, and $\delta:=0$ in Claim~3.25
and derive that
\[
\begin{array}{rl}
  \beta\wedge (\alpha\vee (\beta\wedge\delta))
  &=\Delta\wedge(R_i\vee(\Delta\wedge 0))\\
&=\Delta\wedge R_i\\
&=0\\
&\leq \delta,
\end{array}
\]
and we could indeed conclude from Theorem~2.19~(8)
that $\C C(R_i,\Delta;0)$ holds, and the proof
could be concluded as we did in \cite{shape}.
Unfortunately, in the proof of Theorem~3.24 of \cite{shape},
$R_i$ is only a tolerance and not necessarily a congruence
(it is a reflexive, symmetric, compatible, binary relation, 
but it need not be a transitive relation).
This means that Theorem~2.19~(8) does not apply to the situation.
Theorem~2.19~(7) is more appropriate for tolerances.
Using it, one may correctly derive from Claim~3.25 
(i.e., from the claim
that $\Delta\cap R_i=0$ holds) that $\C C(\Delta,R_i;0)$ holds.
But this is different than $\C C(R_i,\Delta;0)$,
which is what we needed for the rest of the proof of Theorem~3.24.
If we had proved a slightly stronger statement
in Claim~3.25, namely that $\Delta\cap (R_i\circ R_i)=0$, then we
could correctly derive from Theorem~2.19~(7)
that $\C C(R_i,\Delta;0)$ holds.
But, as it happens, there is no way to correct the
proof of Theorem~3.24, since the statement of the
theorem is false, as we will see.
It will be corrected and strengthened in
Theorem~\ref{main} of the next section.

\section{The Main Correction}

Let $\mathcal V$ be a variety.
Let $\mathscr A$ be the property
``for every $\m a\in {\mathcal V}$
and every tolerance $T$ on $\m a$,
if $T$ is abelian, then the congruence
on $\m a$ generated by $T$ is abelian.''
Let ${\mathscr T}$ denote the Maltsev definable property
``there exists a Taylor term''.
Let ${\mathscr W}$ be the Maltsev definable property
``there exists a weak difference term''.
The incorrect result 
\cite[Theorem~3.24]{shape} asserts that
for any variety $\mathcal V$, ${\mathscr T}\Rightarrow {\mathscr A}$.
If true, this would imply that the conjunction of
the properties 
${\mathscr T}$ and ${\mathscr A}$ is equivalent to
the property 
${\mathscr T}$. We shall learn in Theorem~\ref{main} that
the conjunction of
the properties 
${\mathscr T}$ and ${\mathscr A}$ is equivalent to the property
${\mathscr W}$. This is enough to disprove
\cite[Theorem~3.24]{shape}, since it is known that
${\mathscr W}$ is slightly (but properly) stronger that
${\mathscr T}$.

In more detail,
a \emph{Taylor term} for a variety $\mathcal V$
is a $\mathcal V$-term $T(x_1,\ldots,x_{n})$ such that
$\mathcal V$ satisfies the
identity $T(x,\ldots,x)\approx x$
and, for each $i$ between $1$ and $n$, $\mathcal V$ satisfies
some identity of the form $T(\wec{w})\approx T(\wec{z})$
with $w_i \neq z_i$.
A \emph{weak difference term} for $\mathcal V$
is a $\mathcal V$-term $w(x,y,z)$ such that 
for any $\m b\in {\mathcal V}$,
$w^{\m b}(a,a,b) = b=w^{\m b}(b,a,a)$ holds
whenever the pair $(a,b)$ is contained in
an abelian
congruence.
It is known that the class of varieties
with a weak difference term is definable
by a nontrivial idempotent Maltsev condition
(see \cite[Theorem~4.8]{kearnes-szendrei}).
It is known that the class of varieties
with a Taylor term is definable
by the \emph{weakest} nontrivial idempotent Maltsev condition.
(This is derivable from Corollaries~5.2 and 5.3 of
\cite{taylor}.)
In particular, this proves that
${\mathscr W}\Rightarrow {\mathscr T}$.
It is known that
${\mathscr T}\not\Rightarrow {\mathscr W}$
for arbitrary varieties
(see \cite[Example~4.4]{relative}).
It is also known that
${\mathscr W}\Leftrightarrow {\mathscr T}$
for locally finite varieties
(see \cite[Theorem~9.6]{hobby-mckenzie}).
These facts,
(i) ${\mathscr W}\Rightarrow {\mathscr T}$,
(ii) ${\mathscr T}\not\Rightarrow {\mathscr W}$,
(iii) ${\mathscr W}\Leftrightarrow {\mathscr T}$ for locally finite
varieties,
are the data supporting our statement that 
``$\mathscr W$ is slightly (but properly) stronger than $\mathscr T$''.

The next theorem proves that
${\mathscr T}\Rightarrow ({\mathscr W}\Leftrightarrow {\mathscr A})$.
Since
${\mathscr W}\Rightarrow {\mathscr T}$, it is possible to
derive 
${\mathscr W}\Leftrightarrow ({\mathscr T}\;\&\; {\mathscr A})$
from this.
In the terminology of \cite{relative}, this means
that the property
${\mathscr A}$ is Maltsev definable \emph{relative} to
${\mathscr T}$ by the Maltsev condition that
defines $\mathscr W$.

\begin{thm} \label{main}
  Let $\mathcal V$ be a variety that has a Taylor term
(i.e., $\mathcal V$ satisfies ${\mathscr T}$).
The following are equivalent properties for $\mathcal V$:
\begin{enumerate}
\item[(1)]
  $\mathcal V$ has a weak difference term.
  ($\mathcal V$ satisfies ${\mathscr W}$.)
\item[(2)]  $\mathcal V$ has term $w(x,y,z)$
  such that, whenever $\m a\in {\mathcal V}$,
  $T$ is an abelian tolerance on $\m a$, and
  $(a,b)\in T$, then $w^{\m a}(a,a,b) = b = w^{\m a}(b,a,a)$.
  ($w$ is a ``weak difference term for tolerances''.)
\item[(3)] In any algebra of $\mathcal V$,
  any abelian tolerance is a congruence.
\item[(4)]
  In any algebra of $\mathcal V$,
  the congruence generated by an
  abelian tolerance is abelian.
  ($\mathcal V$ satisfies ${\mathscr A}$.)
\end{enumerate}
\end{thm}

\begin{proof}
  Assume that Item (1) holds. The class of
  varieties with a weak difference term
  is definable by an idempotent Maltsev condition,
  say $\Sigma$. Let ${\mathcal V}^{\circ}$
  be the idempotent reduct of $\mathcal V$.
  ${\mathcal V}^{\circ}$ also satisfies $\Sigma$,
  hence must have a weak difference term, say $w(x,y,z)$.
  Now choose $\m a\in {\mathcal V}$, let $T$
  be an abelian tolerance of $\m a$, and choose $(a,b)\in T$.
  Observe that $\{a,b\}^2 = \{(a,a),(a,b),(b,a),(b,b)\}\subseteq T$,
  since $T$ is reflexive and symmetric.
  By Zorn's Lemma, we may extend the subset $\{a,b\}\subseteq A$
to a subset $B\subseteq A$ maximal for (i) $\{a,b\}\subseteq B$
  and (ii) $B^2\subseteq T$ ($B$ is a \emph{$T$-block}).
  The maximality of $B$ and
  the fact that $T$ is a tolerance jointly
  guarantee that $B$ is closed
  under all idempotent term operations of $\m a$.
  Thus, if $\m a^{\circ}$ denotes the idempotent reduct of $\m a$,
  then $B$ is a subuniverse of $\m a^{\circ}$.
  The fact that $T$ is an abelian tolerance of
  $\m a$ implies that $T$ is an abelian tolerance
  of the reduct $\m a^{\circ}$, hence $T|_{B}$
  is an abelian tolerance of $\m b\;(\leq \m a^{\circ})$.
  Since $B^2\subseteq T$, it follows that $T|_B$
  is the universal binary relation on $B$,
  hence $\m b$ is an abelian algebra.
  We have that $w(x,y,z)$ is a weak difference term for
  ${\mathcal V}^{\circ}$, 
  $\m b\in {\mathcal V}^{\circ}$, and
  $\{a,b\}\subseteq B$, so $w^{\m b}(a,a,b)=b=w^{\m b}(b,a,a)$.
  Since $w(x,y,z)$ is a term of $\m a$, we have
  $w^{\m a}(a,a,b)=b=w^{\m a}(b,a,a)$ whenever $(a,b)\in T$.
This establishes that   
  $w(x,y,z)$ is a weak difference term
for tolerances for the variety $\mathcal V$,
proving that Item (2) holds.

  Now assume that Item (2) holds.
  Choose any $\m a\in {\mathcal V}$  and any abelian
  tolerance $T$ of $\m a$. Our goal is to
  prove that $T$ is transitive, so choose
  elements satisfying $a\wrel{T}b\wrel{T}c$.
  Since $(a,b),(b,c)\in T$
  and $T$ is a reflexive subalgebra of $\m a^2$
we have $w^{T}((a,b),(b,b),(b,c))\in T$.
Since $w(x,y,z)$ is a weak difference term for
tolerances and $T$ is an abelian tolerance containing
$(a,b)$ and $(b,c)$, we have
$w^{T}((a,b),(b,b),(b,c))=(w^{\m a}(a,b,b),w^{\m a}(b,b,c))=(a,c)$,
proving that $(a,c)\in T$.
This shows that Item (3) holds.

The implication (3)$\Rightarrow$(4) holds since the
congruence generated by a congruence is itself.

To prove the final implication (4)$\Rightarrow$(1)
we explain instead why
$\neg (1) \Rightarrow \neg (4)$ holds.
Assume that Item (1) fails. According to 
Theorem~4.5 of \cite{relative}, which is stated
in the Introduction of this article, some $\m a\in \mathcal V$
  must have a pentagon in its congruence lattice
labeled as in
Figure~\ref{fig8}, and for this pentagon
we may assume that $[\alpha,\alpha]=0$ 
  and $\C C(\theta,\alpha;\delta)$.
  Let $T = (\delta\circ\beta\circ\delta)/\delta$.
  We will complete the proof that (4) fails by arguing that
  $T$ is an abelian tolerance on $\m a/\delta$
  which does not generate an abelian congruence on $\m a/\delta$.

  For the congruences depicted in Figure~\ref{fig8}, we have
  $[\alpha,\alpha]=0$ and $\beta\leq \alpha$,
  so $[\beta,\beta]=0$
  by monotonicity of the commutator. 
  In terms of the centralizer relation, this
  means $\C C(\beta,\beta;0)$ holds.
  Now $\C C(\beta,\beta;0)$ is equivalent to
  $\C C(\beta,\beta;\delta)$
  by Theorem~2.19~(4) above, so
  $\C C(\beta,\beta;\delta)$ must also hold.
  We also have $\C C(\delta,\beta;\delta)$ by
  Theorem~2.19~(8). The statements
  $\C C(\beta,\beta;\delta)$ and
  $\C C(\delta,\beta;\delta)$ may be combined to
  $\C C(\beta\vee \delta,\beta;\delta)$ by
Theorem~2.19~(5).  
  By Theorem~2.19~(3) we derive that
  $\C C(\beta\vee\delta,\delta\circ \beta\circ \delta;\delta)$ holds.
  If $S$ is the tolerance $\delta\circ\beta\circ \delta$ on $\m a$,
  then since
  $S\subseteq \beta\vee\delta$ we derive from
  $\C C(\beta\vee\delta,\delta\circ \beta\circ \delta;\delta)$
  and Theorem~2.19~(1) that 
  $\C C(S, S; \delta)$ holds.
  Now, by Theorem~2.19~(10), we derive that
  $\C C(S/\delta,S/\delta;0)$ holds.
  But $T=S/\delta$, so $\C C(T,T;0)$ holds.
  We have proved that
  $T$ is an abelian tolerance of $\m a/\delta$.

  We argue now that the congruence generated by $T$
  is not abelian.
  The congruence generated by $T=(\delta\circ \beta\circ \delta)/\delta$
  is its transitive closure,
  which is $(\beta\vee\delta)/\delta=\alpha/\delta$.
  To prove that this congruence on $\m a/\delta$ is not abelian,
  it suffices to prove that
  $\C C(\alpha,\alpha;\delta)$
  fails in $\m a$.
  Since $\beta\leq \alpha$ and $\theta\leq \alpha$,
  and the centralizer relation is monotone in its first
  two variables (Theorem~2.19~(1)),
it suffices to prove that
  $\C C(\beta,\theta;\delta)$
fails in $\m a$.
This is the one place in the proof where we invoke the
fact that $\mathcal V$ has a Taylor term.
It is proved in the first part of
\cite[Theorem~4.16~(2)]{shape}
that if $\m a$ is an algebra in a variety
with a Taylor term which has a pentagon in its
congruence lattice labeled as in Figure~\ref{fig8},
then $\C C(\beta,\theta;\delta)$ cannot hold.
This establishes $\neg (4)$.
\end{proof}

We emphasize the most relevant
aspect of the previous theorem in the form of a corollary.

\begin{cor} \label{main_cor}
  Theorem~3.24 of \cite{shape} is correct if the
  hypothesis that $\mathcal V$ has a Taylor term
  (or ``$\mathcal V$ satisfies a nontrivial idempotent Maltsev condition'')
  is strengthened to the 
  hypothesis that $\mathcal V$ has a weak difference term.  
\end{cor}

\section{Consequences} \label{consequences}

Some of the later results in \cite{shape} depend on the incorrect
Theorem~3.24. We identify them here and indicate
whether they are true as stated, and explain
how to adjust the statements that are not true as stated.

\begin{enumerate}
\item (\cite[Theorem~3.27]{shape})
This theorem examines the following property of a variety $\mathcal V$:
given two perspective
congruence intervals $\alpha/(\alpha\wedge \beta)$
and $(\alpha\vee\beta)/\beta$
of some $\m a\in \mathcal V$,
one interval is abelian if and only the other is abelian.
Call this property $\mathscr B$.
\cite[Theorem~3.27]{shape} asserts that
${\mathscr T}\Rightarrow {\mathscr B}$.
This implication {\bf is false}.
Using Corollary~\ref{main_cor},
we can correct Theorem~3.27 by strengthening
the hypothesis $\mathscr T$ to the
hypothesis $\mathscr W$.
The corrected version then expresses that
${\mathscr W}\Rightarrow {\mathscr B}$.
In fact, it is not hard to see that
${\mathscr W}\Leftrightarrow ({\mathscr T}\;\&\; {\mathscr B})$.
(See Theorem~\ref{M_Def}~(2) below.)

\bigskip

\item
  (\cite[Theorem~4.16]{shape})
  This theorem claims that certain finite lattices
  with some specified centralities
  cannot occur
  as sublattices of congruence lattices
  of algebras in varieties with a Taylor term.
The theorem has three parts and each
of these parts has more than one claim.
Part (1) of the theorem involves three claims.
\begin{itemize}
\item Each of the three claims of 4.16~(1) are {\bf true as stated}.  
\item The first and third of the three claims of 4.16~(2)
  are {\bf true as stated}.  The second claim of 4.16~(2)
  {\bf is false}.
\item The two claims of 4.16~(3)
  are derived from the false claim of 4.16~(2), hence
  {\bf we withdraw these two claims.}
  We have neither proof nor counterexample for these claims.
\item {\bf All claims of all items in Theorem 4.16 are
  true when the hypothesis ``${\mathcal V}$ has a Taylor term''
  is strengthened to ``${\mathcal V}$ has a weak difference term''.}
  (However, the two claims of 4.16~(3) were already known
  to be true in the presence of a weak difference term,
  cf.\ \cite[Corollary 5.8]{lipparini}.)  
\end{itemize}  

We give more detail about the second bullet point.
  Using the labeling in
  Figure~\ref{fig8}, it is stated in \cite[Theorem~4.16~(2)]{shape}
  that a pentagon of congruences
  where $\C C(\beta,\theta;\delta)$ holds
  cannot occur as a sublattice in a variety with a Taylor term.
  The proof
  given for this is correct.
  But then it is stated that
a pentagon of congruences
  where $\C C(\beta,\beta;\beta\wedge \delta)$ holds
  cannot occur as a sublattice.
Call this property ${\mathscr C}$.
  The proof given that ${\mathscr T}\Rightarrow {\mathscr C}$
  depends on \cite[Theorem~3.27]{shape}, which
  (we have seen) is false without a weak difference term.
  In fact, it can be shown by arguments like those above
  that ${\mathscr W}\Leftrightarrow ({\mathscr T}\;\&\; {\mathscr C})$.
  ((See Theorem~\ref{M_Def}~(3) below.))
To summarize, the part of \cite[Theorem~4.16~(2)]{shape}
  that refers to $\C C(\beta,\beta;\beta\wedge \delta)$
  is false as stated, but becomes
  true when the hypothesis ``${\mathcal V}$ has a Taylor term''
  is strengthened to ``${\mathcal V}$ has a weak difference term''.
 
\bigskip

\item
  (\cite[Lemma~6.7]{shape})
  This lemma cites \cite[Theorem~3.24]{shape} in its proof, but
  one of the hypotheses of \cite[Lemma~6.7]{shape} is that
  the variety has a weak difference term.
  In this setting, \cite[Theorem~3.24]{shape} has been shown
  to hold in Corollary~\ref{main_cor} above.
  Consequently, \cite[Lemma~6.7]{shape} is {\bf true as stated}.

\bigskip

\item
  (\cite[Theorem~6.25]{shape})
  The proof of \cite[Theorem~6.25]{shape}
  refers to Theorem~3.27 immediately before Claim 6.26
  to complete part of the argument.
  An alternative argument is suggested in that proof which uses
  Lemma~6.10 instead. This alternative argument
  does not rely on Theorem 3.24.
  Moreover, the hypothesis of 
  \cite[Theorem~6.25]{shape} includes that $\mathcal V$ has a
  weak difference term, and in this setting
  Corollary~\ref{main_cor} allows us to leave the proof
  as it is.
Thus, \cite[Theorem~6.25]{shape} is {\bf true as stated}.  
  
\bigskip

\item
  (\cite[Theorem~8.1]{shape})
This theorem refers to
  \cite[Theorem~3.24]{shape}
  in its proof of (4)$\Rightarrow$(5).
  The reference can be eliminated, as we now
  explain. Item (4) is the assertion
  that $\mathcal V$ contains no algebra
  with a nontrivial abelian congruence.
Item (5) is the assertion
  that $\mathcal V$ contains no algebra
  with a nontrivial abelian tolerance.
  What we need to establish, therefore, is that
  any variety which omits nontrivial abelian 
  congruences also omits nontrivial abelian
  tolerances.
  We can use Theorem~\ref{main} above to do this.
Assume that Item (4) of \cite[Theorem~8.1]{shape} holds.  
Since $\mathcal V$  contains no algebra
  with a nontrivial abelian congruence,
  the ternary projection $w(x,y,z):=x$
  is a weak difference term for $\mathcal V$.
  Any variety with a weak difference term has a Taylor
  term, so Theorem~\ref{main}  applies to $\mathcal V$.
  In particular, from the implication
  (1)$\Rightarrow$(3) of Theorem~\ref{main} 
  we see that an abelian tolerance on an algebra
  in $\mathcal V$ is a congruence.
  Since abelian congruences in $\mathcal V$ are trivial,
abelian tolerances in $\mathcal V$ are also trivial,
so Item (5) of \cite[Theorem~8.1]{shape} holds.
This shows that \cite[Theorem~8.1]{shape} is {\bf true as stated}.  
\end{enumerate}

\bigskip

\noindent
{\bf Summary.}
\begin{enumerate}
\item  (\cite[Theorem~3.24]{shape}) is {\bf false as stated}.
Theorem~3.24 is {\bf true
when the hypothesis ``${\mathcal V}$ has a Taylor term''
  is strengthened to ``${\mathcal V}$ has a weak difference term''.}
\item (\cite[Theorem~3.27]{shape}) is {\bf false as stated}.
  Theorem~3.27 is {\bf true
when the hypothesis ``${\mathcal V}$ has a Taylor term''
  is strengthened to ``${\mathcal V}$ has a weak difference term''.}
\item (\cite[Theorem~4.16]{shape})
  \begin{itemize}
\item Each of the three claims of 4.16~(1) are {\bf true as stated}.  
\item The first and third of the three claims of 4.16~(2)
  are {\bf true as stated}.  The second claim of 4.16~(2)
  {\bf is false}.
\item We {\bf do not know whether the two claims of 4.16~(3) are true}.
\item {\bf All claims of all items in Theorem 4.16 are
  true when the hypothesis ``${\mathcal V}$ has a Taylor term''
  is strengthened to ``${\mathcal V}$ has a weak difference term''.}
\end{itemize}
\item (\cite[Lemma~6.7]{shape}) is {\bf true as stated}.  
\item (\cite[Theorem~6.25]{shape}) is {\bf true as stated}.  
\item (\cite[Theorem~8.1]{shape})  is {\bf true as stated}.  
\end{enumerate}
\bigskip

\section{Afterword} \label{after}

We close this note by recording 
the results of the form
${\mathscr W}\Leftrightarrow ({\mathscr T}\;\&\; {\mathscr X})$,
where ${\mathscr X} = {\mathscr A}, {\mathscr B}$ or ${\mathscr C}$,
which we have considered in this paper,
in the form of a theorem statement.
Results of the form
${\mathscr W}\Leftrightarrow ({\mathscr T}\;\&\; {\mathscr X})$,
assert that Property $\mathscr X$
is Maltsev definable relative to
$\mathscr T$ (= existence of a Taylor term)
by the Maltsev condition that defines $\mathscr W$
(= existence of a weak difference term).

Observe that ${\mathscr W}~\Leftrightarrow~({\mathscr T}\;\&\; {\mathscr X})$
is equivalent to the conjunction of
(i)~${\mathscr W}~\Rightarrow~{\mathscr T}$,
(ii)~${\mathscr W}~\Rightarrow~{\mathscr X}$, and
(iii)~$({\mathscr T}\;\&\; {\mathscr X})~\Rightarrow~{\mathscr W}$.
Implication (i) is known to be true:
the class of varieties
with a weak difference term is definable
by a nontrivial idempotent Maltsev condition
(cf.\ \cite[Theorem~4.8]{kearnes-szendrei}) and
the class of varieties
with a Taylor term is definable
by the weakest nontrivial idempotent Maltsev condition
(cf.\ Corollaries~5.2 and 5.3 of \cite{taylor}.)
Implication (iii) may be rewritten in the form
(iii)' $({\mathscr T}\;\&\;\neg{\mathscr W})\Rightarrow \neg {\mathscr X}$.
Altogether, we want to record why 
\begin{enumerate}
\item[(ii)] ${\mathscr W}~\Rightarrow~{\mathscr X}$, and
\item[(iii)'] $({\mathscr T}\;\&\; \neg{\mathscr W})\Rightarrow \neg {\mathscr X}$
\end{enumerate}  
hold when ${\mathscr X} = {\mathscr A}, {\mathscr B}$ or ${\mathscr C}$.

\begin{thm} \label{M_Def}
If $\mathcal V$ has a Taylor term,
then any one of the following properties 
will hold for $\mathcal{V}$ if and only if
$\mathcal{V}$ has a weak difference term.
\begin{enumerate}
\item[(1)] {\rm ($\mathcal V$ satisfies ${\mathscr A}$.)}
  In any algebra of $\mathcal V$,
  the congruence generated by an
  abelian tolerance is abelian.
\item[(2)] {\rm ($\mathcal V$ satisfies ${\mathscr B}$.)}
Given two perspective
congruence intervals $\alpha/(\alpha\wedge \beta)$
and $(\alpha\vee\beta)/\beta$
of some $\m a\in \mathcal V$,
one interval is abelian if and only the other is abelian.
\item[(3)] {\rm ($\mathcal V$ satisfies ${\mathscr C}$.)}
  A pentagon of congruences, labeled as in 
    Figure~\ref{fig9},
    cannot occur as a sublattice of $\Con(\m a)$
    if $\m a\in \mathcal{V}$ and 
$\C C(\beta,\beta;\beta\wedge \delta)$ holds.
\begin{figure}[ht]
\begin{center}
\setlength{\unitlength}{1mm}
\begin{picture}(20,33)
\put(0,15){\circle*{1.2}}
\put(10,0){\circle*{1.2}}
\put(10,30){\circle*{1.2}}
\put(20,10){\circle*{1.2}}
\put(20,20){\circle*{1.2}}

\put(10,0){\line(-2,3){10}}
\put(10,0){\line(1,1){10}}
\put(10,30){\line(-2,-3){10}}
\put(10,30){\line(1,-1){10}}
\put(20,20){\line(0,-1){10}}

\put(-4.5,13){$\beta$}
\put(22,9){$\delta$}
\put(22,19){$\theta$}
\put(9,32){$\alpha$}
\put(6,-5){$\beta\wedge\delta$}
\end{picture}
\bigskip

\caption{\sc Forbidden sublattice if 
  $\C C(\beta,\beta;\beta\wedge\delta)$ holds.}\label{fig9}
\end{center}
\end{figure}

\end{enumerate}
\end{thm}

\begin{proof}
Item~(1) is Theorem~\ref{main} $(1)\Leftrightarrow(4)$.

It was noted above in Consequence~(1)
that the erroneous proof in \cite{shape}
that $\mathscr T\Rightarrow \mathscr B$
is nevertheless
valid as a proof of
$\mathscr W\Rightarrow \mathscr B$. This gives us
(ii) ${\mathscr W}~\Rightarrow~{\mathscr B}$, as desired.
To prove (iii)', that is, to prove
$({\mathscr T}\;\&\; \neg{\mathscr W})\Rightarrow \neg {\mathscr B}$,
assume that $\mathcal{V}$ has a Taylor term but not a
weak difference term.
By Theorem~4.5~(5) of \cite{relative}, there exists
$\m a\in \mathcal{V}$ with a pentagon
in its congruence lattice labeled as in Figure~\ref{fig8}.
Since $[\alpha,\alpha]=0$ in this pentagon,
we get that $[\beta,\beta]=0$ by monotonicity.
Equivalently, the congruence quotient
$\beta/(\beta\wedge\delta)$ is abelian.
For the purpose of obtaining a contradiction
assume that $\mathcal{V}$ satisfies $\mathscr B$.
This allows us to deduce that the perspective interval
$(\beta\vee\delta)/\delta = \alpha/\delta$ is also
abelian. This is equivalent to the statement that 
$\C C(\alpha,\alpha;\delta)$ holds, and therefore
we derive that $\C C(\beta,\theta;\delta)$ holds
from Theorem~2.19 of \cite{shape} (i.e., by the monotonicity of
the centralizer relation in the ﬁrst two variables). 
But according to Theorem~4.2 of \cite{relative},
no such pentagon can exist in the congruence lattice
of any algebra with a Taylor term. This contradiction establishes
Item~(2). 

To prove Item~(3), 
we assert that 
that the erroneous proof in \cite{shape} of Theorem~4.16~(2)~(i),
which claims 
that $\mathscr T\Rightarrow \mathscr C$,
is nevertheless
valid as a proof of
$\mathscr W\Rightarrow \mathscr C$.
Indeed, the proof of
$\mathscr T\Rightarrow \mathscr C$ from
\cite{shape} is actually a valid proof
that $\mathscr B\Rightarrow \mathscr C$
coupled together with an erroneous
reference to
$\mathscr T\Rightarrow \mathscr B$.
We have argued above that $\mathscr W\Rightarrow \mathscr B$,
so coupling this with the valid proof
that $\mathscr B\Rightarrow \mathscr C$ yields
$\mathscr W\Rightarrow \mathscr C$.

To complete the proof of this theorem,
we must prove (iii)'
$({\mathscr T}\;\&\; \neg{\mathscr W})\Rightarrow \neg {\mathscr C}$.
Assume that $\mathcal{V}$ fails $\mathscr W$, i.e. $\mathcal{V}$
does not have a weak difference term.
By Theorem 4.5~(5)  of \cite{relative}
we know that there exists $\m a\in \mathcal{V}$
with congruences labeled as in Figure~\ref{fig8}
such that $[\alpha,\alpha]=0=\beta\wedge\delta$ 
and $\C C(\theta,\alpha;\delta)$.
Since $[\alpha,\alpha]=0=\beta\wedge\delta$ we have
$\C C(\alpha,\alpha;\beta\wedge\delta)$
by the definition of the commutator,
hence $\C C(\beta,\beta;\beta\wedge\delta)$
by the monotonicity of the centralizer in its first two variables.
Thus, $\m a$ has a
pentagon of congruences that violate $\mathscr C$.
This establishes that $\neg \mathscr W\Rightarrow \neg \mathscr C$,
so the weaker implication
$(\mathscr T\;\&\;\neg \mathscr W)\Rightarrow \neg \mathscr C$ also holds.\footnote{We have actually established that $\mathscr W\Leftrightarrow \mathscr C$.}
\end{proof}
  
\section*{Acknowledgement}
  We thank Ralph Freese for comments on an early draft
  of this paper and the referee for comments on the
  succeeding drafts.

\bigskip
  
\section*{Declarations}

\medskip

\noindent
{\bf Ethical approval.} Not applicable.

  
\noindent
{\bf Competing interests.} Both authors are on the editorial board of Algebra Universalis.

  
\noindent
{\bf Authors' contributions.} 
All authors contributed equally.

  
\noindent
{\bf Availability of data and materials.}
No new data were created or analyzed during this study.
Data sharing is not applicable to this article,
as no datasets were generated or used.

  
\noindent
{\bf Funding.}
Not applicable.

\bibliographystyle{plain}

\end{document}